\documentclass[a4paper,11pt,reqno,twoside]{article}
\usepackage{amssymb}
\usepackage[frenchb,english]{babel}
\usepackage{amscd}
\usepackage{amsmath,amsthm,amsfonts,amssymb,graphicx}
\usepackage[colorlinks,linkcolor=blue,anchorcolor=blue,citecolor=green]{hyperref}
\usepackage{mathrsfs}
\usepackage{fancyhdr}
\usepackage{subfigure}
\usepackage{bm}
\usepackage{multicol}
\usepackage{picins}
\usepackage{abstract}
\usepackage{graphicx}

\makeatother

\begin{document}
\theoremstyle{plain}
\newtheorem{Definition}{Definition}[section]
\newtheorem{Proposition}{Proposition}[section]
\newtheorem{Property}{Property}[section]
\newtheorem{Theorem}{Theorem}[section]
\newtheorem{Lemma}[Theorem]{\hspace{0em}\bf{Lemma}}
\newtheorem{Corollary}[Theorem]{Corollary}
\newtheorem*{Remark}{Remark}

\baselineskip 15pt

\noindent  {\LARGE The Kobayashi pseudometric for the Fock-Bargmann-Hartogs domain \\\\
 and its application}\\\\

\noindent\text{Enchao Bi$^{1}$\; \&  \; Guicong Su$^{2}$ \; \&  \; Zhenhan Tu$^{2*}$}\\

 \noindent\small {${}^1$ School of Mathematics and Statistics, Qingdao University, Qingdao, Shandong 266071, P.R. China}\\

 \noindent\small {${}^2$ School of Mathematics and Statistics, Wuhan
University, Wuhan, Hubei 430072, P.R. China } \\

\noindent\text{Email: bienchao@whu.edu.cn (E. Bi),\;
suguicong@whu.edu.cn (G. Su),\;
zhhtu.math@whu.edu.cn (Z. Tu)}

\renewcommand{\thefootnote}{{}}
\footnote{\hskip -16pt {$^{*}$Corresponding author. \\ } }
\\

\normalsize \noindent\textbf{Abstract}\quad {The Fock-Bargmann-Hartogs domain $D_{n,m}$
 in $\mathbb{C}^{n+m}$ is defined by the inequality $\|w\|^2<e^{-\|z\|^2},$  where $(z,w)\in \mathbb{C}^n\times \mathbb{C}^m$, which is an unbounded non-hyperbolic domain
 in $\mathbb{C}^{n+m}$.
This paper mainly consists of three parts. Firstly, we give the explicit expression of geodesics of $D_{n,1}$  in the sense of
Kobayashi pseudometric; Secondly, using the formula of geodesics, we calculate explicitly the Kobayashi pseudometric on $D_{1,1}$; Lastly, we establish the Schwarz lemma at the boundary for holomorphic mappings between the nonequidimensional  Fock-Bargmann-Hartogs domains by using the formula for the Kobayashi pseudometric on $D_{1,1}$.
\vskip 10pt

\vskip 10pt

  \textperiodcentered \, 32H02  \textperiodcentered \, 30C80

\section{Introduction}
Let $\mathbb{C}^n$ be the $n$-dimensional complex Hilbert space with the inner product and the norm given by
\begin{equation*}
\langle z,w\rangle=\sum_{j=1}^{n}z_j\overline{w_j},\ \ \ \Vert z\Vert=(\langle z,z\rangle)^{\frac{1}{2}}
\end{equation*}
where $z,w\in\mathbb{C}^{n}$. Throughout this paper, we write a point $z\in\mathbb{C}^n$ as a row vector in the following $1\times n$ matrix form
\begin{equation*}
z=(z_1, \ldots, z_n)
\end{equation*}

By $E$ we denote the unit disk in $\mathbb{C}$. Let $D$ be a domain in $\mathbb{C}^n$. For $(z,
\zeta)\in D\times\mathbb{C}^n$ we define
\begin{Definition}
The Kobayashi pseudometric on $D$ is the function $\mathcal{K}_{D}:D\times\mathbb{C}^n\rightarrow\mathbb{R}^{+}\cup\{0\}$ defined by
\begin{equation}\label{eq1}
\mathcal{K}_{D}(z,\zeta)=\inf\limits_{f}\big\{\vert\alpha\vert: \exists f : E\rightarrow D\   holomorphic, f(0)=z,f'(0)\alpha=\zeta\big\}.
\end{equation}
\end{Definition}

\begin{Definition}
A holomorphic mapping $\varphi:E\rightarrow D$ is said to be a $\mathcal{K}_{D}$-geodesic for $(z,\zeta)$ if $\varphi$ is a function achieving the minimum in \eqref{eq1}.
\end{Definition}

 It is well known that if $D$ is a taut domain then for any $(z,\zeta)\in D\times \mathbb{C}^n$, there exists a $\mathcal{K}_{D}$-geodesic for any $(z,\zeta)$ (see \cite{Jarnicki-Plug}).
 We know, if $D$ is convex domain, then any $\mathcal{K}_{D}$-geodesic $\varphi$ for $(z,\zeta)$ with $\zeta\neq 0$ is a $\mathcal{K}_{D}$-geodesic for any $(\varphi(\lambda),\varphi^{'}(\lambda))$ $(\lambda\in E)$, and moreover, the Kobayashi pseudometric and the Carath\'{e}dory pseudometric on $D$ coincide (see \cite{Lempert}). In this case any $\mathcal{K}_{D}$-geodesic for some $(z,\zeta)$ with $\zeta\neq 0$ is called a complex geodesic (see \cite{Ves}) also.

 In 1994, Pflug-Zwonek \cite{Plug-Zwonek} considered  complex ellipsoids $\varepsilon(p)$ (a class of weakly pesudoconvex domains) defined by
 \begin{equation*}
 \varepsilon(p) :=\big\{\vert z_1\vert^{2p_1}+\ldots+\vert z_n\vert^{2p_n}<1\big\}\subset\mathbb{C}^n,\;(n\geq 2)
 \end{equation*}
where $p=(p_1,\ldots,p_n)$ with $p_j>0$ $(1\leq j\leq n)$. It is well known that the complex ellipsoids are taut domains, and, they are convex if and only if $p_j\geq\frac{1}{2}$ for $j=1,\ldots,n$. Moreover, $\partial\varepsilon(p)$ is $\mathcal{C}^{\omega}$ and strongly pseudoconvex at all boundary points $z\in(\partial\varepsilon(p))\cap(\mathbb{C}_{*})^n$ (where $\mathbb{C}_{*}=\mathbb{C}\backslash \{0\}$). Pflug-Zwonek \cite{Plug-Zwonek} gave a necessary condition for $\mathcal{K}_{\varepsilon(p)}$-geodesic in complex ellipsoids for all $p_j>0$ $(1\leq j\leq n)$ as follows.

 \begin{Theorem}[see \cite{Plug-Zwonek}]\label{Th1.1}
 Let $\varphi: E\rightarrow \varepsilon(p)$ be a $\mathcal{K}_{\varepsilon(p)}$-geodesic for $(\varphi(0),\varphi'(0))$ with $\varphi'(0)\neq0$, where $\varphi_j\not\equiv0$ $(1\leq j\leq n)$. Then we have
 \begin{equation}
 \varphi_j(\lambda)=B_j(\lambda)\bigg(a_j\frac{1-\overline{\alpha_j}\lambda}{1-\overline{\alpha_0}\lambda}\bigg)^{\frac{1}{p_j}},
 \end{equation}
 where $B_j(\lambda)$ is the Blaschke product and $\alpha_j, \alpha_0, a_j$ fulfill the following relations
 \begin{equation}\label{eq3}
 a_j\in\mathbb{C_{*}}, \alpha_j\in\overline{E}\; (1\leq j\leq n) \textrm{ and } \alpha_0\in E,
 \end{equation}
 \begin{equation}\label{eq4}
 \alpha_0=\sum\limits_{j=1}^n\vert a_j\vert^2\alpha_j,\; 1+\vert \alpha_0\vert^2=\sum\limits_{j=1}^n\vert a_j\vert^2(1+\vert \alpha_j\vert^2).
 \end{equation}
 Moreover, if $p_j\geq\frac{1}{2}$ for some $j$, then we can choose either $B_j\equiv1$ or $B_j(\lambda)=\frac{\lambda-\alpha_j}{1-\overline{\alpha_j}\lambda}$ with $\vert \alpha_j\vert<1$.\\
Additionally, if $\vert \alpha_j\vert<1$ for all $1\leq j\leq n$, then either $B_j\equiv1$ or $B_j(\lambda)=\frac{\lambda-\alpha_j}{1-\overline{\alpha_j}\lambda}$ for all  $j=1,\ldots,n$.
 \end{Theorem}

Also they proved the geodesic is almost proper (refer to \cite{Plug-Zwonek}), that is,
\begin{Corollary}[see \cite{Plug-Zwonek}]\label{Le1.2}
 Let $\varphi: E\rightarrow \varepsilon(p)$ be a $\mathcal{K}_{\varepsilon(p)}$-geodesic for $(\varphi(0),\varphi'(0))$ with $\varphi'(0)\neq0$. Then $\varphi^{*}(\partial E)\subseteq \partial \varepsilon(p)$ where $\varphi^{*}$ denote the boundary value of $\varphi$.
\end{Corollary}

In 1992, Blank-Fan-Klein-Krantz-Ma-Pang \cite{B}  delivered an effective formula of the Kobayashi metric in the convex ellipsoids $\varepsilon(p)$ for $p=(1,m)$ (i.e., $m \geq 1/2$).
By using the condition for $\mathcal{K}_{\varepsilon(p)}$-geodesic (i.e., Theorem \ref{Th1.1}),
Pflug-Zwonek \cite{Plug-Zwonek} obtained the formulas of the Kobayashi metric in the non-convex complex ellipsoids $\varepsilon(p)$ for $p=(1,m)$ (i.e., $m < 1/2$) in 1994.

The Fock-Bargmann-Hartogs domain $D_{n,m}$ is defined by
\begin{equation}
D_{n,m}:=\big\{(z,w)\in \mathbb{C}^{n} \times \mathbb{C}^m: \Vert w\Vert^{2}<e^{-{\Vert z\Vert}^{2}}\big\}.
\end{equation}
The Fock-Bargmann-Hartogs domains $D_{n,m}$ are strongly pseudoconvex with smooth real-analytic boundary. We note that each $D_{n,m}$ contains $\{(z, 0)\in \mathbb{C}^n\times\mathbb{C}^m\} \cong \mathbb{C}^n$. Thus each $D_{n,m}$ is not hyperbolic in the sense of Kobayashi and $D_{n,m}$ can not be biholomorphic to any bounded domain in $\mathbb{C}^{n+m}$. Therefore, each Fock-Bargmann-Hartogs domain $D_{n,m}$ is an unbounded non-hyperbolic domain in $\mathbb{C}^{n+m}$.

As we know, in the case that a domain in $\mathbb{C}^{n}$ is unbounded, or more generically, non-hyperbolic strongly pseudoconvex (in the sense that the Levi form for the defining function is strictly positive for each boundary point, see section 2.1 in \cite{KYZ}),
we can no longer expect that the geometric and analytic properties of the domain is
as good as in the bounded case. Therefore, the study of Fock-Bargmann-Hartogs domain $D_{n,m}$ attracts lots of attentions recently.
Yamamori \cite{Y} gave an explicit formula for the Bergman kernel of the Fock-Bargmann-Hartogs domain $D_{n,m}$ in terms of the polylogarithm functions in 2013. In 2014, Kim-Ninh-Yamamori \cite{Kim-Ninh-Yamamori} determined the full holomorphic automorphisms of the Fock-Bargmann-Hartogs domain $D_{n,m}$ and it turns out that the automorphism group is non-compact and the domain $D_{n,m}$ isn't homogeneous.
In 2015, Tu-Wang \cite{TW} proved the rigidity of proper holomorphic mappings between two equidimensional Fock-Bargmann-Hartogs domains, which implies that any proper holomorphic self-mapping on the Fock-Bargmann-Hartogs domain $D_{n,m}$ with $m\geq 2$ must be an automorphism.
In 2016, Bi-Feng-Tu \cite{BFT} obtained an explicit formula for the Bergman kernel of the weighted Hilbert space of
square integrable holomorphic functions on $D_{n,m}$, and furthermore, use the explicit expression to prove the existence of balanced metrics for a class of Fock-Bargmann-Hartogs domains.

Recently, Kim-Yamamori-Zhang \cite{KYZ} studied the invariant metrics on  $D_{n,m}$, and obtained the Bergman and
K\"ahler-Einstein metrics on  $D_{n,m}$  are metrically equivalent, and, determined the comparisons among
the Carath\'eodory and Kobayashi pseudometrics on $D_{n,m}$ also.  But, since  Kim-Yamamori-Zhang \cite{KYZ} do not know
the explicit form of geodesic on  $D_{1,1}$ in general (see the appendix of \cite{KYZ}), the Kobayashi pseudometric on $D_{1,1}$ cannot be calculated explicitly in Kim-Yamamori-Zhang \cite{KYZ}.

In fact, the $\mathcal{K}_{D_{n,m}}$-geodesics for the Fock-Bargmann-Hartogs domain $D_{n,m}$ are quite different from the bounded complete Reinhardt domains, such as \cite{Vis1} and \cite{Vis2}, or from the minimal ball in $\mathbb{C}^n$ (\cite{P-Y}). However, by using the decomposition theorem for a mapping $f\in\mathcal{H}^\infty$, $f\not\equiv 0$, and the equivalent definition for stationary (e.g., see \cite{Pa}), this paper gives all $\mathcal{K}_{D_{n,1}}$-geodesics  on  $D_{n,1}$  are necessarily of the following form:
\begin{Theorem}\label{Th1.3}
Let $\varphi=(\varphi_1,\ldots,\varphi_{n+1}): E\rightarrow D_{n,1}$ be a $\mathcal{K}_{D_{n,1}}$-geodesic for $(\varphi(0),\varphi'(0))$ with
$\varphi'(0)\neq 0$, $\varphi_i\not\equiv 0$ $(1\leq i\leq n+1)$. Then we have
\begin{equation}
\varphi_j(\lambda)=\frac{\big(\frac{\lambda-\alpha_j}{1-\overline{\alpha_j}\lambda}\big)^{r_j}ia_j(1-\overline{\alpha_j}\lambda)}{(1-\overline{\alpha_0}
\lambda)-\big(\frac{\lambda-\alpha_{n+1}}{1-\overline{\alpha_{n+1}}\lambda}\big)^{r_{n+1}}a_{n+1}(1-\overline{\alpha_{n+1}}\lambda)},\;\; 1\leq j\leq n,
\end{equation}
and
\begin{equation}
\varphi_{n+1}(\lambda)=B(\lambda)\exp\bigg\{-\frac{1}{2}\frac{1-\overline{\alpha_0}
\lambda+\big(\frac{\lambda-\alpha_{n+1}}{1-\overline{\alpha_{n+1}}\lambda}\big)^{r_{n+1}}a_{n+1}(1-\overline{\alpha_{n+1}}\lambda)}{1-\overline{\alpha_0}
\lambda-\big(\frac{\lambda-\alpha_{n+1}}{1-\overline{\alpha_{n+1}}\lambda}\big)^{r_{n+1}}a_{n+1}(1-\overline{\alpha_{n+1}}\lambda)}\bigg\},
\end{equation}
where $r_j\in\{0,1\}$ and $r_j=1$ implies $\alpha_j\in E$. Moreover $\alpha_j,\;\alpha_0,\;a_j$ fulfill \eqref{eq3} and \eqref{eq4}.\\
Additionally, either $B(\lambda)\equiv 1$ or $B(\lambda)=\frac{\lambda-\gamma}{1-\overline{\gamma}\lambda}$ where $\gamma$ satisfies that if $r_{n+1}=1$, then $\gamma=\frac{\alpha_0+\overline{a_{n+1}}}{1+\overline{a_{n+1}\alpha_{n+1}}}$, and if $r_{n+1}=0$, then $\gamma=\frac{\alpha_0-\overline{a_{n+1}}\alpha_{n+1}}{1-\overline{a_{n+1}}}$
\end{Theorem}

As an easy consequence of Theorem \ref{Th1.3} we get:
\begin{Corollary}\label{Le1.4}
Let $\varphi=(\varphi_1,\ldots,\varphi_{n+1}): E\rightarrow D_{n,1}$ be a $\mathcal{K}_{D_{n,1}}$-geodesic for $(\varphi(0),\varphi'(0))$ with
$\varphi'(0)\neq 0$. Then $\varphi$ extends smoothly onto the closure $\overline{E}$, and $\varphi(\partial E)\subseteq \partial D_{n,1}$.
\end{Corollary}

Next we will give the explicit formula for the Kobayashi pseudometric on the domain $D_{1,1}$. Note that Kim-Ninh-Yamamori \cite{Kim-Ninh-Yamamori} completely describe the group of holomorphic automorphisms for the Fock-Bargmann-Hartogs domains $D_{n,m}$ as follows.
\begin{Theorem}[see \cite{Kim-Ninh-Yamamori}]
The automorphism group $\mathrm{Aut}(D_{n,m})$ is generated by all the following automorphisms of $D_{n,m}$:
\begin{equation*}
\varphi_{U}:\;(z,w)\rightarrow (zU,w),\;U\in\mathcal{U}(n);
\end{equation*}
\begin{equation*}
\varphi_{V}:\;(z,w)\rightarrow(z,wV),\;V\in \mathcal{U}(m);
\end{equation*}
\begin{equation*}
\varphi_{a}:\;(z,w)\rightarrow (z-a,e^{\langle z,a\rangle-\frac{1}{2}{\Vert a\Vert}^{2}}w),\;a\in\mathbb{C}^{n},
\end{equation*}
where $\mathcal{U}(n)$ denotes the set of the $n\times n$ unitary matrices.
\end{Theorem}
Therefore, in order to find the formulas for the Kobayashi pseudometric on $D_{1,1}$, it suffices to calculate the Kobayashi pseudometric for $((0,b),(X,Y))$ $(0\leq b<1)$ by the invariance of the Kobayashi pseudometric under the automorphism.

In the case $b = 0$ or $X = 0$, we can find the formulas for the Kobayashi pseudometric for $((0,b),(X,Y))$  $(0\leq b<1)$  as usual. But,
in the case $0<b<1$ and $X\neq 0$, we need much more calculations to find the formulas for the Kobayashi pseudometric for $((0,b),(X,Y))$  $(0\leq b<1)$. To make the calculations simpler, we define
\begin{equation}\label{def000}
v :=\frac{\vert Y\vert^2}{\vert X\vert^2}\;\;(X\not= 0).
\end{equation}

Then we have the formulas for the Kobayashi pseudometric as follows:
\begin{Theorem}\label{Th1.6}
The Kobayashi pseudometric of $D_{1,1}$ for $((0,b),(X,Y))$ can be expressed as follows:\\
$(i)$ When $b=0$, then, for $Y\neq 0$, we have
$$\mathcal{K}_{D_{1,1}}((0,0),(X,Y))=\mu(X,Y),$$
where $\mu(X,Y)$ is the Minkowski functional of $D_{1,1}$ which is uniquely determined by
$$\frac{\vert X\vert^2}{\mu^2}+\ln{\vert Y\vert^2}-2\ln \mu=0.$$
Additionally, we have
$$\mathcal{K}_{D_{1,1}}((0,0),(X,0))=0.$$
$(ii)$ When $0<b<1$, $X=0$, then we have
$$\mathcal{K}_{D_{1,1}}^2((0,b),(0,Y))=\frac{\vert Y\vert^2}{1-b^2}+\frac{\vert bY \vert^2}{(1-b^2)^2}.$$
$(iii)$ When $0<b<1$, $X\neq0$, then we have
\begin{equation}\label{eq22}
\mathcal{K}_{D_{1,1}}^2((0,b),(X,Y))
=\left\{
\begin{aligned}
&-\frac{1}{2\ln b}\bigg(\vert X\vert^2-\frac{\vert Y\vert^2}{2b^2\ln b}\bigg)     &     \mbox{for $v<4b^2$},\\
&\frac{\vert X\vert^2}{\alpha(1-b^2e^\alpha)}                                     &     \mbox{for $0<\alpha\leq\beta,\;v\geq 4b^2$},\\
&-\frac{1}{2\ln b}\bigg(\vert X\vert^2-\frac{\vert Y\vert^2}{2b^2\ln b}\bigg)     &     \mbox{for $\beta<\alpha<-2\ln b,\;v\geq4b^2$},
\end{aligned}
\right.
\end{equation}
where $\alpha,\;v$ satisfy
$$\frac{b^4\alpha^2e^\alpha+(1-b^2e^\alpha)^2e^{-\alpha}}{\alpha(1-b^2e^\alpha)}=v-2b^2$$
and $\beta$ is the only solution of the interval $(0,-2\ln b)$ for the equation
$$-\frac{1}{2\ln b}\beta(1-b^2e^\beta)e^\beta+\frac{1}{4b^2\ln^2 b}\big(b^2\beta e^\beta+(1-b^2e^\beta)\big)^2-e^\beta=0.$$
\end{Theorem}

The classical Schwarz lemma gives information about the behaviour of a holomorphic function on the disc at the origin, subject only to the relatively mild
hypotheses that the function map the disc to the disc and the origin to the origin. There are far-reaching generalizations of the classical Schwarz lemma, due to Ahlfors \cite{A} and others (e.g., see Kim-Lee \cite{KL} and references
therein). It is natural to consider various boundary versions of Schwarz lemma. There is a classical Schwarz lemma at the boundary as follows:

\begin{Theorem} [see \cite{Garnett}]\label{Garnett} Let $f : E \rightarrow E$ be a holomorphic function on the open unit disk $E$. If f is holomorphic at $z = 1$ with $f (1) = 1$, then
$f'(1)\geq \frac{\vert1-\overline{f(0)}\vert^2}{1-\vert f(0)\vert^2}>0.$ Moreover, the inequality is sharp.
\end{Theorem}

Chelst \cite{C} and Osserman \cite{O} studied the Schwarz lemma at the boundary of the
unit disk also. Burns-Krantz \cite{BK}, Huang  \cite{H} and  Krantz \cite{Krantz} explored versions of the Schwarz lemma at the boundary point of some domains in $\mathbb{C}^n$. In 2015, Liu-Wang-Tang \cite{Tang-Liu2} gave a new type of Schwarz lemma at the boundary of the unit ball in $\mathbb{C}^n$ as follows.

\begin{Theorem}[see Liu-Wang-Tang \cite{Tang-Liu2}]\label{Tang-Liu2}
Let ${\bf{B}}^n:= \{z=(z_1,\cdots,z_n) \in \mathbb{C}^n: |z_1|^2 +\cdots + |z_n|^2 < 1\}$ be the open unit ball. Let $ f : \bf{B}^n \rightarrow  \bf{B}^n$ be a holomorphic mapping, and
the complex Jacobian matrix of $f$ at $a\in \bf{B}^n $ is denoted by $J_f(a)=(\frac{\partial f_i}{\partial z_j}(a))_{n\times n}$.
If $f$ is holomorphic
at $z_0\in \partial \bf{B}^n$ and $f(z_0) = z_0$, then the following statements hold.\\
$(i)$ The normal vector $z_0^T$ to $\partial \bf{B}^n$ at $z_0$ is an eigenvector of $\overline{J_f(z_0)}^T$ and the corresponding eigenvalue $\lambda$ is a real number (that is, $ \overline{J_f(z_0)}^T z_0^T = \lambda z_0^T$) with
\begin{equation}
\lambda\geq\frac {\vert1-\overline{f(0)}z_0^T\vert^2} {1-\vert f(0)\vert^2}>0,
\end{equation}
Thus the real number $\lambda$ is also an eigenvalue of ${J_f(z_0)}$. \\
$(ii)$ For any other eigenvalues $\mu_j$ of $J_f(z_0)$, there exist $\alpha^j\in T_p^{1,0}(\partial \bf{B}^n)\setminus\{0\}$ such that $\alpha^{j}J_f(z_0)^T=\mu_j\alpha^{j}$ ($j=2,\cdots, n$). Moreover, for all
$j=2,\cdots, n$,
\begin{equation}
\vert \mu_j\vert\leq \sqrt\lambda.
\end{equation}
Moreover, the inequalities are sharp.
\end{Theorem}


When $n = 1$ and $z_0 = 1$, Theorem \ref{Tang-Liu2} reduces to Theorem \ref{Garnett}, which extends the boundary Schwarz lemma to high dimensions. By using the boundary behavior of the Carath\'{e}dory and Kobayashi metrics on bounded strongly pseudoconvex domains with smooth boundary (see Graham \cite{Graham}), recently, Liu-Tang \cite{Tang-Liu} established the new type of boundary Schwarz lemma for holomorphic self-mappings of bounded strongly pseudoconvex domain in $\mathbb{C}^n$. The similar result can be found in Bracci-Zaitsev \cite{BZ} by using quite different method. Following the idea in Liu-Tang \cite{Tang-Liu} , this paper will use our formulas for the Kobayashi pseudometric on $D_{1,1}$ to establish the boundary Schwarz lemma for holomorphic mappings between the nonequidimensional Fock-Bargmann-Hartogs domains  $D_{1,1}$ and $D_{n,m}$.

Actually, by the homogeneity of the boundary of $D_{n,m}$  under the automorphism (see Kim-Ninh-Yamamori \cite{Kim-Ninh-Yamamori}), we will fix the boundary points $p=(0,1)\in \partial D_{1,1}$ and $q=(0,\ldots,0,1,0,\ldots,0)\in \partial D_{n,m}$, i.e., $z=(0,\ldots,0)\in \mathbb{C}^{n}$ and $w=(1,0,\ldots,0)\in \mathbb{C}^{m}$. We give the boundary Schwarz lemma for holomorphic mappings between the nonequidimensional Fock-Bargmann-Hartogs domains as follows.

\begin{Theorem}\label{Th1.8}
Let $F=(f,h):D_{1,1}\rightarrow D_{n,m}$ be a holomorphic mapping and holomorphic at $p$ with $F(p)=q$. Then we have the result as follows.  \\
$(i)$ There exist $\lambda\in\mathbb{R}$ such that
\begin{equation*}
\overline{J_F(p)}^{T}q^{T}=\lambda p^T
\end{equation*}
with $\lambda\geq\vert1-\overline{h_1(0)}\vert^2/(1-\vert h_1(0)\vert^2)>0$. Notice that $p^{T}$ and $q^{T}$ are the normal vectors to the boundary of $D_{1,1}$ at $p$ and $D_{n,m}$ at $q$ respectively.\\
$(ii)$ ${J_F(p)}$ can be regarded as a linear operator from $T_p^{1,0}(\partial D_{1,1})$ to $T_{F(p)}^{1,0}(\partial D_{n,m})$. Moreover, we have
\begin{equation*}
\Vert J_F(p)\Vert_{op}\leq\sqrt\lambda,
\end{equation*}
where $\Vert\cdot\Vert_{op}$ means the usual operator norm.
\end{Theorem}
\begin{Remark}
When $n=1$ and $m=1$, Theorem \ref{Th1.8} can be obtained by Propostion 1.1 in  Bracci-Zaitsev $\cite{BZ}$.
\end{Remark}

Our paper is organised as follows. Firstly, we will give the necessary forms of the geodesic for $D_{n,1}$  in the
sense of Kobayashi pseudometric; Secondly, using this forms, we will calculate  explicitly the Kobayashi pseudometric of $D_{1,1}$; Lastly, we prove the Schwarz lemma at the boundary for holomorphic mappings between the nonequidimensional Fock-Bargmann-Hartogs domains by using the formulas for the Kobayashi pseudometric on $D_{1,1}$.

\section{The geodesic for $D_{n.1}$}
Although $D_{n,1}$ is an unbounded domain, we can also obtain a similar result of Lemma 8 in \cite{Plug-Zwonek} as follows:
\begin{Lemma}\label{Le2.2}
Suppose that $\varphi=(\varphi_1,\ldots, \varphi_{n+1}): E\rightarrow D_{n,1}$ is a $\mathcal{K}_{D_{n,1}}$-geodesic for $(\varphi(0),\varphi'(0))$ with
$\varphi'(0)\neq 0$. Then we have

$(1)$ $\varphi_{n+1}(\lambda)$ can be decomposed as
$$\varphi_{n+1}(\lambda)=B(\lambda)A(\lambda)$$
where $B(\lambda)$ is a Blaschke product and $A(\lambda)$ is a nowhere vanishing function from $H^\infty(E)$.

$(2)$ Moreover, let $\mathcal{Z}$ be the zeros of $B(\lambda)$ and  denote $\mathcal{\widetilde{Z}}$ a subset of $\mathcal{Z}$. Let us associate $\mathcal{\widetilde{Z}}$ with the Blaschke product $\widetilde{B}$. Consider the following mapping
$$\widetilde{\varphi}:=(\varphi_1, \ldots, \varphi_n, \widetilde{B}A).$$
If $\widetilde{\varphi}$ is non-constant, then $\widetilde{\varphi}$ is a $\mathcal{K}_{D_{n,1}}$-geodesic for $(\widetilde{\varphi}(0),\widetilde{\varphi}'(0))$ and $\widetilde{\varphi}'(0)\neq0$.
\end{Lemma}

\begin{proof}[Proof]
Since $\varphi=(\varphi_1, \ldots, \varphi_{n+1}): E\rightarrow D_{n,1}$ implies that $\vert\varphi_{n+1}(\lambda)\vert<1$.
Hence $\varphi_{n+1}(\lambda)\in H^\infty(E)$. Therefore, we get the conclusion $(1)$ in Lemma \ref{Le2.2}  by the decomposition theorem (see \cite{Garnett}).

Now, we will give the conclusion $(2)$ in Lemma \ref{Le2.2}. We firstly prove that $\widetilde{\varphi}(E)\subseteq D_{n,1}$. We do the proof by repeating inductively the following procedure. We divide $\varphi_{n+1}(\lambda)$ by the Blaschke factor assigned to a zero $\mathcal{Z}\backslash \mathcal{\widetilde{Z}}$. We proceed so till we have exhausted all the set $\mathcal{Z}\backslash \mathcal{\widetilde{Z}}$. That means $\widetilde{\varphi}$ can be obtained by the above procedure. Let $\rho$ be a subharmonic function defined by
$$\rho(z,w)=\vert z_1\vert^2+\ldots+\vert z_n\vert^2+\ln {\vert w\vert}^2.$$

In view of the maximum principle for the subharmonic functions applied to the composition of the mappings obtained from $\varphi$ after a finite number of steps of the above mentioned procedure with $\rho$, consequently the limit function of the composition is not larger than $0$. We can see that this composition is not larger than $0$ on $E$. That means $\rho\circ \widetilde{\varphi}\leq0$ on $E$. Since $\widetilde{\varphi}$ is non-constant and $D_{n,1}$ is strongly pseudoconvex, it follows that $\rho\circ \widetilde{\varphi}<0$ on $E$.

If $\widetilde{\varphi}$ is not a $\mathcal{K}_{D_{n,1}}$-geodesic for $(\widetilde{\varphi}(0),\widetilde{\varphi}'(0))$, then there exists a map
$\widetilde{\psi}=(\widetilde{\psi}_1,\ldots, \widetilde{\psi}_{n+1})$ such that
$$\widetilde{\psi}(0)=\widetilde{\varphi}(0),\; \widetilde{\psi}'(0)=\widetilde{\varphi}'(0),\;\widetilde{\psi}(E)\subset\subset D_{n,1},$$

Consider the mapping $\psi=(\psi_1,\ldots, \psi_{n+1})$ where $\psi_j$ are defined by
$$\psi_j(\lambda)=\widetilde{\psi_j}(\lambda),\;\psi_{n+1}(\lambda)=\frac{B(\lambda)}{\widetilde{B}(\lambda)}\widetilde{\psi}_{n+1}(\lambda).$$

It follows
$$\psi(0)=\varphi(0),\;\psi'(0)=\varphi'(0)$$
and moreover $\psi(E)\subset\subset D_{n,1}$.  This contradicts the fact that $\varphi=(\varphi_1,\ldots,\varphi_{n+1}): E\rightarrow D_{n,1}$ be a $\mathcal{K}_{D_{n,1}}$-geodesic for $(\varphi(0),\varphi'(0))$.

Lastly, we prove that $\widetilde{\varphi}'(0)\neq0$. In fact, if $\widetilde{\varphi}'(0)=0$, then the map $\phi(\lambda)=(\phi_1,\ldots,\phi_{n+1})$ where $$\phi_j(\lambda)=\varphi_j(0),\; \phi_{n+1}(\lambda)=\frac{B(\lambda)}{\widetilde{B}(\lambda)}\widetilde{\varphi}_{n+1}(0)$$
is also a $\mathcal{K}_{D_{n,1}}$-geodesic for $(\varphi(0),\varphi'(0))$. However $\phi(\lambda)\subset\subset D_{n,1}$. This is a contradiction.
\end{proof}

Using Lemma \ref{Le2.2}, we can explicitly describe the $\mathcal{K}_{D_{n,1}}$-geodesic as follows.

\begin{proof}[Proof of Theorem \ref{Th1.3}]
Firstly, by Lemma \ref{Le2.2}, we know that $\widetilde{\varphi}$ is a $\mathcal{K}_{D_{n,1}}$-geodesic for $(\widetilde{\varphi}(0),\widetilde{\varphi}'(0))$ where $\widetilde{\varphi}(\lambda)=(\varphi(\lambda), A(\lambda))$ and $A(\lambda)$ is nowhere vanishing function under the condition that $\widetilde{\varphi}$ is not a constant. Now we consider the following map
$$g_j(\lambda)=\varphi_j(\lambda),\; g_{n+1}(\lambda)=-2i\ln {A(\lambda)}.$$
Then $g=(g_1,\ldots, g_{n+1})$ is a $\mathcal{K}_{\Omega}$-geodesic for $(g(0),g'(0)),$  where $\Omega$ is defined by
$$\Omega:=\big\{(z,w)\in\mathbb{C}^{n+1}:\text{Im } w>\vert z_1\vert^2+\ldots+\vert z_n\vert^2\big\}.$$
Otherwise, there will be a mapping $f(\lambda)=(f_1(\lambda),\ldots, f_{n+1}(\lambda))$ fulfilling $f(0)=g(0)$, $f'(0)=g'(0)$ and $f(E)\subset\subset \Omega$. Combined with the definition of $g(\lambda)$, it is not hard to see that the map $\widetilde{f}=(\widetilde{f}_1,\ldots, \widetilde{f}_{n+1})$ defined by
$$\widetilde{f}_j(\lambda)=f_j(\lambda),\;\widetilde{f}_{n+1}(\lambda)=e^{-\frac{1}{2i}f_{n+1}(\lambda)}$$
maps $E$ into $D_{n,1}$. A simple computation implies that
$$\widetilde{f}(0)=\widetilde{\varphi}(0),\;\widetilde{f}'(0)=\widetilde{\varphi}'(0)$$
and $\widetilde{f}(E)\subset\subset D_{n,1}$, a contradiction.

 Now we consider the new mapping $h(\lambda)=(h_1(\lambda),\ldots, h_{n+1}(\lambda))$ derived from composition $g$ with the M\"{o}bius transformation, that is,
 $$h_j(\lambda)=\frac{2g_j(\lambda)}{g_{n+1}(\lambda)+i},\;h_{n+1}{(\lambda)}=\frac{g_{n+1}(\lambda)-i}{g_{n+1}(\lambda)+i}.$$
Then $h(\lambda)$ is a $\mathcal{K}_{\bf{B}^{n+1}}$-geodesic for $(h(0),h'(0))$ with $h'(0)\neq0$, where $\bf{B}^{n+1}\subseteq \mathbb{C}^{n+1}$ denotes the open unit ball. So we get that
$$h_j(\lambda)=\bigg(\frac{\lambda-\alpha_j}{1-\overline{\alpha_j}\lambda}\bigg)^{r_j}\bigg(a_j\frac{1-\overline{\alpha_j}\lambda}{1-\overline{\alpha_0}
\lambda}\bigg),\;\; 1\leq j\leq n+1$$
by Theorem \ref{Th1.1}, where $r_j\in\{0,1\}$ and $r_j=1$ implies $\alpha_j\in E$. Here $\alpha_i,\;a_j$ satisfies \eqref{eq3} and \eqref{eq4}.

Thus, for $1\leq j\leq n$, we have
$$g_j(\lambda)=\frac{ih_j(\lambda)}{1-h_{n+1}(\lambda)},\; g_{n+1}(\lambda)=\frac{i(1+h_{n+1}(\lambda))}{1-h_{n+1}(\lambda)}.$$

It follows that
\begin{equation}\label{gai 1}
\varphi_j(\lambda)=\frac{\big(\frac{\lambda-\alpha_j}{1-\overline{\alpha_j}\lambda}\big)^{r_j}ia_j(1-\overline{\alpha_j}\lambda)}{(1-\overline{\alpha_0}
\lambda)-\big(\frac{\lambda-\alpha_{n+1}}{1-\overline{\alpha_{n+1}}\lambda}\big)^{r_{n+1}}a_{n+1}(1-\overline{\alpha_{n+1}}\lambda)},
\end{equation}
and
\begin{equation}\label{gai 2}
A(\lambda)=\exp\bigg\{-\frac{1}{2}\frac{1-\overline{\alpha_0}
\lambda+\big(\frac{\lambda-\alpha_{n+1}}{1-\overline{\alpha_{n+1}}\lambda}\big)^{r_{n+1}}a_{n+1}(1-\overline{\alpha_{n+1}}\lambda)}{1-\overline{\alpha_0}
\lambda-\big(\frac{\lambda-\alpha_{n+1}}{1-\overline{\alpha_{n+1}}\lambda}\big)^{r_{n+1}}a_{n+1}(1-\overline{\alpha_{n+1}}\lambda)}\bigg\}.
\end{equation}

If $\widetilde{\varphi}$ is a constant, then we can also derive same expressions with some parameters equaling to $0$. Then it remains to determine the formula for the Blaschke product $B(\lambda)$ in $\varphi_{n+1}$.

In order to determine the formula for the Blaschke product $B(\lambda)$ in $\varphi_{n+1}$, we will use the notation ``stationary" to complete our proof. In the sequel, we mainly focus our attention on a smooth strongly pseudoconvex domain $D\subseteq \mathbb{C}^n$.
\begin{Lemma}(See \cite{Pa})
Let $D$ be a smooth strongly pseudoconvex domain in $\mathbb{C}^n$. Suppose that $\varphi\in \textrm{Hol}(E,D)\cap C^k(\overline{E},\overline{D})$ with $\varphi(\partial E)\subseteq \partial D$ $(k\geq2)$ and $\varphi$ is a proper emdedding. Then for all $\lambda\in \partial E$, the function defined by
\begin{equation}
\lambda\mapsto \lambda \frac{\partial \rho}{\partial z}(\varphi(\lambda))\bullet \varphi'(\lambda)
\end{equation}
is a positive $C^{k-1}$ function, where $\rho$ denote the defining function of $D$, and
$$w\bullet z= w_1z_1+\ldots+w_nz_n$$
for $z=(z_1,\ldots,z_n),\;w=(w_1,\ldots,w_n)$.
\end{Lemma}
\begin{Definition}\label{Def3.1} (See \cite{Pa})
With the assumptions above, we will denote by $p\in C^{k-1}(\partial E)$ the positive function defined by
\begin{equation*}
p^{-1}(\lambda)=\lambda \frac{\partial \rho}{\partial z}(\varphi(\lambda))\bullet \varphi'(\lambda).
\end{equation*}
\end{Definition}
With the help of function $p$, we introduce the dual map $\widetilde{\varphi}$ of $\varphi$ as follows.
\begin{Definition}(See \cite{Pa})
Let $\widetilde{\varphi}$ be a $C^{k-1}$ function on $\partial E$ defined by
\begin{equation*}
\widetilde{\varphi}(\lambda):=p(\lambda)\lambda \frac{\partial \rho}{\partial z}(\varphi(\lambda)),\;\; \lambda\in \partial E.
\end{equation*}
\end{Definition}
\begin{Definition}(See \cite{Pa})
A mapping $\varphi$ is said to be stationary if $\varphi$ is a smooth embedding of $\overline{E}$ into $\overline{D}$, holomorphic on $E$ such that $\varphi(\partial E)\subseteq \partial D$ and $\widetilde{\varphi}$ extends to a smooth mapping on $\overline{E}$, holomorphic on $E$.
\end{Definition}
\begin{Lemma}\label{Le3.2}(See \cite{Pa})
Let $D$ be a smooth strongly pseudoconvex domain in $\mathbb{C}^n$. Suppose that $\varphi:\overline{E}\rightarrow \overline{D}$ is a $C^\infty$ mapping with $\varphi(\partial E)\subseteq \partial D$, $\varphi(0)\in D$, which is holomorphic on $E$. Then if $\varphi$ is a $\mathcal{K}_{D}$-geodesic for $(\varphi(0),\varphi'(0))$, then $\varphi$ is stationary.
\end{Lemma}
\begin{Remark}
Actually, under the assumption that $\varphi$ is $C^\infty$ on $\overline{E}$, Lemma \ref{Le3.2} holds for a smooth strongly pseudoconvex domain $D$, whether $D$ is bounded or not.
\end{Remark}
\begin{Lemma}\label{Le3.3}(See \cite{G})
Assume that $f\in H^{1}(E)$ is a mapping such that
$$\frac{1}{\lambda}f^*(\lambda)>0$$
for almost $\lambda \in \partial E$. Then there exists a $r>0$  and $\alpha\in \overline{E}$ such that
$$f(\lambda)=r(\lambda-\alpha)(1-\overline{\alpha}\lambda),\; \lambda \in E.$$
\end{Lemma}

In view of Lemma \ref{Le2.2} it suffices to discuss the case when the Blaschke product of $\varphi_{n+1}$ has at most a finite number of zeros, then under the assumption of finite zeros, we get the result as follows.
\begin{Lemma}\label{Le2.4}
Let $\varphi=(\varphi_1,\ldots,\varphi_{n+1}): E\rightarrow D_{n,1}$ be a $\mathcal{K}_{D_{n,1}}$-geodesic for $(\varphi(0),\varphi'(0))$, and assume that $\varphi_{n+1}$ has at most a finite number of zeros. Then $\varphi$ extends smoothly onto the closure $\overline{E}$ and $\varphi(\partial E)\subseteq \partial D_{n,1}$.
\begin{proof}[Proof of Lemma \ref{Le2.4}]
Now we only need to consider two cases, i.e., $r_{n+1}=0$ and $r_{n+1}=1$. Without loss of generality we can assume that $\varphi(0)=(0,b)$, then we have $\alpha_j=0$ for $1\leq j\leq n$.

When $r_{n+1}=1$, we will prove that the denominators appearing in $\varphi_j$ will never be zero for $\lambda\in \partial E$. In fact, the denominators equal to $0$ mean that
\begin{equation}\label{eq8}
(1-\overline{\alpha_0} \lambda)-(\lambda-\alpha_{n+1})a_{n+1}=0.
\end{equation}
Combining with \eqref{eq4}, we can get
$$1+a_{n+1}\alpha_{n+1}=\lambda a_{n+1}(\overline{a_{n+1}}\overline{\alpha_{n+1}}+1).$$
If there exists $\lambda_0\in \partial E$ such that \eqref{eq8} holds, then it is easy to see
\begin{equation*}
\vert1+a_{n+1}\alpha_{n+1}\vert\times \vert\vert a_{n+1}\vert-1 \vert=0,
\end{equation*}
which implies that
$$1+a_{n+1}\alpha_{n+1}=0, \textrm{ or } \vert a_{n+1}\vert-1=0.$$
If $1+a_{n+1}\alpha_{n+1}=0$, then $\alpha_0=-\overline{a_{n+1}}$. Hence we have $a_j=0$ by \eqref{eq4} for $1\leq j\leq n$. This contradicts the fact $a_j\in \mathbb{C}_{*}$. Similarly, if $\vert a_{n+1}\vert-1=0$, we can also obtain $a_j=0$ for $1\leq j\leq n$. This is a contradiction. Hence for any $\lambda \in \partial E$, \eqref{eq8} will never hold.

Now we consider the case $r_{n+1}=0$. Then the denominators are equal to $0$ if and only if
\begin{equation}\label{eq9}
(1-\overline{\alpha_0}
\lambda)-a_{n+1}(1-\overline{\alpha_{n+1}}\lambda)=0.
\end{equation}
This means that
$$1-a_{n+1}=a_{n+1}\overline{\alpha_{n+1}}\lambda(\overline{a_{n+1}}-1).$$
If there exists $\lambda_{0}\in \partial E$ such that \eqref{eq9} holds, then we have
$$\vert 1-a_{n+1}\vert(\vert a_{n+1}\overline{\alpha_{n+1}}\vert-1)=0.$$
Therefore we obtain
$$a_{n+1}=1,\textrm{ or }\vert a_{n+1}\overline{\alpha_{n+1}}\vert=1.$$
If $a_{n+1}=1$, then \eqref{eq3} and \eqref{eq4} imply that $\alpha_0=\alpha_{n+1}$ and $a_j=0$ for $1\leq j\leq n$, a contradiction. If $\vert a_{n+1}\overline{\alpha_{n+1}}\vert=1$, then we can also get $a_j=0$ for $1\leq j\leq n$. This is also a contradiction. In conclusion, the mapping $\varphi$ can be extended smoothly onto $\overline{E}$.

In the following, we will show that $\varphi(\partial E)\subseteq \partial D_{n,1}$. Notice that $\vert B(e^{i\theta})\vert=1$ for $B(\lambda)$ has at most a finite number of zeros. Together with the form $\varphi_j(\lambda)$ in \eqref{gai 1} and $A(\lambda)$ in \eqref{gai 2}, we conclude that $\varphi(\partial E)\subseteq \partial D_{n,1}$.
\end{proof}
\end{Lemma}

Now we will finish the proof for Theorem \ref{Th1.3}. Let $\varphi=(\varphi_1,\ldots, \varphi_{n+1}): E\rightarrow D_{n,1}$ be a $\mathcal{K}_{D_{n,1}}$-geodesic for $(\varphi(0),\varphi'(0))$, and by Lemma \ref{Le2.2} we can assume that $\varphi_{n+1}$ has at most a finite number of zeros. Then by Lemma \ref{Le3.2} and \ref{Le2.4}, we conclude that $\varphi$ is stationary. Therefore we have for $1\leq j\leq n$,
$$\frac{1}{\lambda}\widetilde{\varphi_j}(\lambda)\varphi_j(\lambda)=p(\lambda)\vert \varphi_j(\lambda)\vert^2>0,\; \lambda \in \partial E,$$
$$\frac{1}{\lambda}\widetilde{\varphi_{n+1}}(\lambda)\varphi_{n+1}(\lambda)=p(\lambda)>0,\; \lambda \in \partial E.$$
Then Lemma \ref{Le3.3} implies that there exists a $s_j>0$  and $\gamma_j\in \overline{E}$ $(1\leq j\leq n+1)$ such that
$$\widetilde{\varphi_j}(\lambda)\varphi_j(\lambda)=s_j(\lambda-\gamma_j)(1-\overline{\gamma_j}\lambda),\; \lambda \in E.$$
It follows that $\varphi_j$ has at most one zero $\gamma_j$ in $E$. Therefore we get $\alpha_j=\gamma_j$ for $1\leq j\leq n$. Combining the above equations, then we have that for $\lambda\in\partial E$,
$$s_{n+1}\vert1-\overline{\gamma_{n+1}}\lambda \vert^2\vert \varphi_j(\lambda)\vert^2=s_j\vert1-\overline{\gamma_j}\lambda \vert^2.$$
This yields that for $\lambda\in\partial E$,
\begin{equation}
s_{n+1}\vert1-\overline{\gamma_{n+1}}\lambda \vert^2\vert a_j\vert^2\vert1-\overline{\alpha_j}\lambda\vert^2=s_j\vert1-\overline{\gamma_j}\lambda \vert^2
\vert(1-\overline{\alpha_0}
\lambda)-\big(\frac{\lambda-\alpha_{n+1}}{1-\overline{\alpha_{n+1}}\lambda}\big)^{r_{n+1}}a_{n+1}(1-\overline{\alpha_{n+1}}\lambda)\vert^2.
\end{equation}
If $r_{n+1}=1$, it follows that
\begin{equation*}
\gamma_{n+1}=\frac{\alpha_0+\overline{a_{n+1}}}{1+\overline{a_{n+1}\alpha_{n+1}}}.
\end{equation*}
Similarly, if $r_{n+1}=0$, then we have
\begin{equation*}
\gamma_{n+1}=\frac{\alpha_0-\overline{a_{n+1}}\alpha_{n+1}}{1-\overline{a_{n+1}}}.
\end{equation*}
This completes the proof for Theorem \ref{Th1.3}.
\end{proof}

\section{The Kobayashi pseudometric on $D_{1,1}$}

In this section, we mainly compute the Kobayashi pseudometric on $D_{n,m}$ with $n=m=1$. Because of the invariance of the Kobayashi pseudometric under the biholomorphic mappings, we just need to show the explicit form of Kobayashi pseudometric at $((0,b),(X,Y))$. Firstly we assume that $b>0$ and $X\neq 0$. Now we consider the mappings of the following forms
\begin{equation}\label{eq16}
\varphi(\lambda)=\bigg(\frac{ia_1\lambda}{(1-\overline{\alpha_0}
\lambda)-a_2(1-\overline{\alpha_2}\lambda)},\exp\bigg\{-\frac{1}{2}\frac{(1-\overline{\alpha_0}
\lambda)+a_2(1-\overline{\alpha_2}\lambda)}{(1-\overline{\alpha_0}
\lambda)-a_2(1-\overline{\alpha_2}\lambda)}\bigg\} \bigg),
\end{equation}
\begin{equation}\label{eq17}
\varphi(\lambda)=\bigg(\frac{ia_1\lambda}{(1-\overline{\alpha_0}
\lambda)-a_2(1-\overline{\alpha_2}\lambda)},\bigg(\frac{\lambda-\frac{\alpha_0-\overline{a_2}\alpha_2}{1-\overline{a_2}}}
{1-\overline{\frac{\alpha_0-\overline{a_2}\alpha_2}{1-\overline{a_2}}}\lambda}\bigg)\exp\bigg\{-\frac{1}{2}\frac{(1-\overline{\alpha_0}
\lambda)+a_2(1-\overline{\alpha_2}\lambda)}{(1-\overline{\alpha_0}
\lambda)-a_2(1-\overline{\alpha_2}\lambda)}\bigg\} \bigg),
\end{equation}
\begin{equation}\label{eq18}
\varphi(\lambda)=\bigg(\frac{ia_1\lambda}{(1-\overline{\alpha_0}
\lambda)-a_2(\lambda-\alpha_2)},\exp\bigg\{-\frac{1}{2}\frac{(1-\overline{\alpha_0}
\lambda)+a_2(1-\overline{\alpha_2}\lambda)}{(1-\overline{\alpha_0}
\lambda)-a_2(\lambda-\alpha_2)}\bigg\} \bigg),
\end{equation}
\begin{equation}\label{eq19}
\varphi(\lambda)=\bigg(\frac{ia_1\lambda}{(1-\overline{\alpha_0}
\lambda)-a_2(\lambda-\alpha_2)},\bigg(\frac{\lambda-\frac{\alpha_0-\overline{a_2}\alpha_2}{1-\overline{a_2}}}
{1-\overline{\frac{\alpha_0-\overline{a_2}\alpha_2}{1-\overline{a_2}}}\lambda}\bigg)\exp\bigg\{-\frac{1}{2}\frac{(1-\overline{\alpha_0}
\lambda)+a_2(1-\overline{\alpha_2}\lambda)}{(1-\overline{\alpha_0}
\lambda)-a_2(\lambda-\alpha_2)}\bigg\} \bigg),
\end{equation}
such that
\begin{equation*}
\varphi(0)=(0,b),\; \tau \varphi'(0)=(X,Y),
\end{equation*}
where $\alpha_0,\;\alpha_j,\;a_j$ fulfill the relations \eqref{eq3} and \eqref{eq4}.


In the sequel, we will compute $\varphi(0)$ and $\varphi'(0)$ case by case. Firstly, if $\varphi$ is of the form \eqref{eq16}, then we have
\begin{equation*}
\left\{
\begin{aligned}
b&=\exp{\frac{-(1+a_2)}{2(1-a_2)}},\\
X&=\frac{\tau ia_1}{1-a_2},\\
Y&=\tau\frac{a_2(\overline{\alpha}_2-\overline{\alpha}_0)}{(1-a_2)^2}\cdot\exp{\frac{-(1+a_2)}{2(1-a_2)}}.
\end{aligned}
\right.
\end{equation*}
Direct computations imply that
\begin{equation}\label{eq20}
\tau^2=-\frac{1}{2\ln b}\bigg(\vert X\vert^2-\frac{\vert Y\vert^2}{2b^2\ln b}\bigg)
\end{equation}
by \eqref{eq3} and \eqref{eq4}.

If $\varphi$ is of the form \eqref{eq17}, then we obtain
\begin{equation*}
\left\{
\begin{aligned}
b&=-\exp{\frac{-(1+a_2)}{2(1-a_2)}}\cdot\frac{\alpha_0-\overline{a_2}\alpha_2}{1-\overline{a_2}},\\
X&=\frac{\tau ia_1}{1-a_2},\\
Y&=\tau\bigg(-\exp{\frac{-(1+a_2)}{2(1-a_2)}}\cdot\frac{\alpha_0-\overline{a_2}\alpha_2}{1-\overline{a_2}}
\frac{a_2(\overline{\alpha}_2-\overline{\alpha}_0)}{(1-a_2)^2}+
(1-\frac{\vert\alpha_0-\overline{a_2}\alpha_2\vert^2}{\vert1-\overline{a_2}\vert^2})\exp{\frac{-(1+a_2)}{2(1-a_2)}}\bigg).
\end{aligned}
\right.
\end{equation*}
After a complicated computation, we have
\begin{equation}\label{eq21}
\tau^2=\frac{\vert X\vert^2}{\alpha(1-b^2e^\alpha)},\;(v\geq 4b^2)
\end{equation}
by \eqref{eq3} and \eqref{eq4}, where $v=\frac{\vert Y\vert^2}{\vert X\vert^2}$ and $\alpha\in (0,-2\ln b)$ is the solution of the following equation
$$\frac{b^4t^2e^t+(1-b^2e^t)^2e^{-t}}{t(1-b^2e^t)}=v-2b^2.$$

If $\varphi$ is of the form \eqref{eq18}, similar to the above arguments, we still have
$$\tau^2=-\frac{1}{2\ln b}\bigg(\vert X\vert^2-\frac{\vert Y\vert^2}{2b^2\ln b}\bigg).$$

If $\varphi$ is of the form \eqref{eq19}, analogously, we obtain
$$\tau^2=\frac{\vert X\vert^2}{\alpha(1-b^2e^\alpha)},\;(v\geq 4b^2),$$
where $v=\frac{\vert Y\vert^2}{\vert X\vert^2}$ and $\alpha\in (0,-2\ln b)$ is the solution of the following equation
$$\frac{b^4t^2e^t+(1-b^2e^t)^2e^{-t}}{t(1-b^2e^t)}=v-2b^2.$$

Therefore, in order to obtain the explicit form of Kobayashi pseudometric, we only need to compare \eqref{eq20} with \eqref{eq21}.

\begin{proof}[Proof of Theorem \ref{Th1.6} ]
When $b=0$, it is well known that $\mathcal{K}_{D_{1,1}}((0,0),(X,Y))$ is Minkowski functional $\mu(X,Y)$ with $\vert Y\vert\neq 0$ which is uniquely determined by
$$\frac{\vert X\vert^2}{\mu^2}+\ln{\vert Y\vert^2}-2\ln \mu=0,$$
while $\mathcal{K}_{D_{1,1}}((0,0),(X,0))=\mathcal{K}_{\mathbb{C}}(0,X)=0$. Then we prove $(i)$.

For the case $X=0$, $b\neq 0$, we first notice the fact that
$$\{0\}\times {\mathbb{B}} \subseteq D_{1,1}\subseteq \mathbb{C}\times\mathbb{B}.$$
Hence we have
$$\mathcal{K}_{B}(b,Y)\leq\mathcal{K}_{D_{1,1}}((0,b),(0,Y))\leq\mathcal{K}_{B}(b,Y).$$
Then we prove $(ii)$.

If $v<4b^2$, then the forms of $\varphi(\lambda)$ is only \eqref{eq16} and \eqref{eq18}. Hence we have
$$\mathcal{K}_{D_{1,1}}^2((0,b),(X,Y))=-\frac{1}{2\ln b}\bigg(\vert X\vert^2-\frac{\vert Y\vert^2}{2b^2\ln b}\bigg).$$

If $v\geq4b^2$, we should compare \eqref{eq20} with \eqref{eq21}. Let $g(t)$ be a function on $[0,-2\ln b]$ defined as follows.
\begin{equation*}
g(t)=-\frac{1}{2\ln b}t(1-b^2e^t)e^t+\frac{1}{4b^2\ln^2 b}\big(b^2te^t+(1-b^2e^t)\big)^2-e^t.
\end{equation*}

It is not hard to see that
$$g(0)=\frac{(1-b^2)^2}{4b^2\ln^2 b}-1>0,\;g(-2\ln b)=0$$
for $b\in (0,1)$. Furthermore, the derivative of $g(t)$ can be computed explicitly as follow
\begin{equation*}
g'(t)=\frac{e^t}{2\ln^2 b}\big(-\ln b\cdot(t+1)+e^t(2b^2t\ln b+b^2\ln b+b^2t^2-b^2t)+t-2\ln^2b\big).
\end{equation*}

Let $\phi(t)$ be defined by
\begin{equation*}
\phi(t)=-\ln b\cdot(t+1)+e^t(2b^2t\ln b+b^2\ln b+b^2t^2-b^2t)+t-2\ln^2b.
\end{equation*}

Then we have
\begin{equation*}
\phi'(t)=1-\ln b+b^2e^t\big(t^2+(2\ln b+1)t+3\ln b-1\big).
\end{equation*}

It follows
\begin{equation*}
\phi''(t)=b^2e^t\big(t^2+(2\ln b+3)t+5\ln b\big),\;t\in[0,-2\ln b].
\end{equation*}

Then we have $\phi'(t)$ is monotone decreasing on $[0,\frac{-(2\ln b+3)+\sqrt{(2\ln b-2)^2+5}}{2}]$ and monotone increasing on$[\frac{-(2\ln b+3)+\sqrt{(2\ln b-2)^2+5}}{2},-2\ln b]$. It is easy to see that $\phi'(-2\ln b)=0$ and
$$\phi'(0)=1-\ln b+b^2(3\ln b-1)>0.$$

Hence there exists only one $\alpha_0$ on $[0,-2\ln b]$ such that $\phi'(\alpha_0)=0$. That means $\phi(t)$ is monotone increasing on $[0,\alpha_0]$ and monotone decreasing on $[\alpha_0,-2\ln b]$. Since $\phi(0)=-\ln b+b^2\ln b-2\ln^2b<0$ and $\phi(-2\ln b)=0$, we can find a $\alpha_1\in (0,\alpha_0)$ such that $\phi(\alpha_1)=0$. This yields that $g(t)$ is monotone decreasing on $[0,\alpha_1]$ and monotone increasing on $[\alpha_1,-2\ln b]$. Combining with the fact $g(0)>0$ and $g(-2\ln b)=0$, we conclude there exists only one $\beta\in (0,\alpha_1)$ such that $g(\beta)=0$ and therefore $g(t)>0$ on $[0,\beta]$ and $g(t)<0$ on $[\beta,-2\ln b]$.

\begin{figure}[h]
\begin{center}
\scalebox{1.0}{\includegraphics[height=90mm]{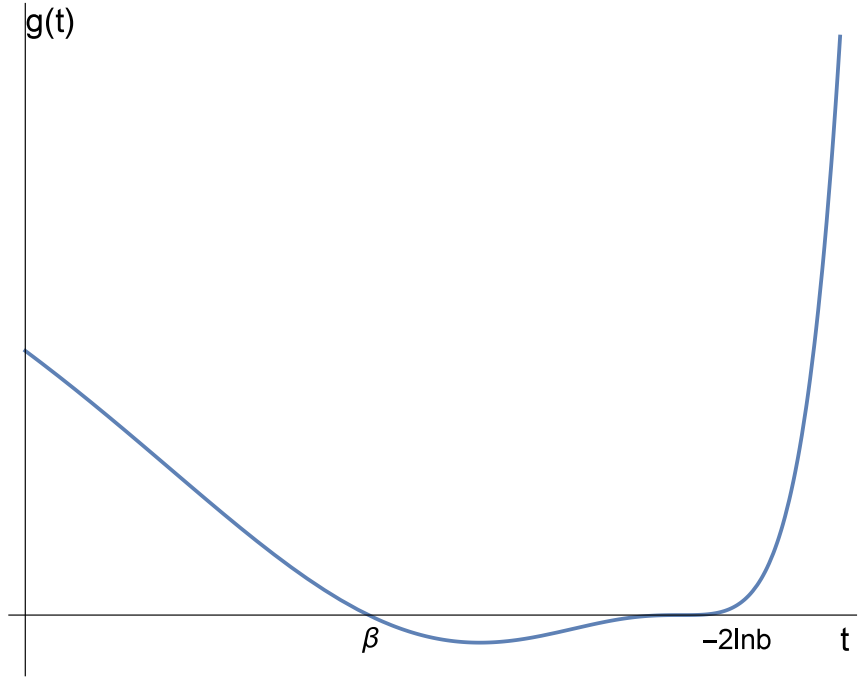}}
\vskip 4mm
\end{center}
\caption{The figure for $g(t)$}
\end{figure}
The proof is complete. \end{proof}

\newpage

\section{Boundary Schwarz Lemma}

The Fock-Bargmann-Hartogs domain $D_{n,m}$
 in $\mathbb{C}^{n+m}$ is defined by the inequality
 $$\rho_{D_{n,m}}(z,w):=\|w\|^2-e^{-\|z\|^2}<0,\;\;  (z,w)\in \mathbb{C}^n\times \mathbb{C}^m.$$
Let $p\in\partial D_{n,m}$, the tangent space $T_p({\partial D_{n,m}})$ to $\partial D_{n,m}$ at $p$ is defined by
\begin{equation*}
T_p(\partial D_{n,m})=\{(\alpha,\beta)\in\mathbb{C}^{n+m},\textrm{Re}\left(\sum_{k=1}^m\overline{w}_k\beta_k+
\sum\limits_{j=1}^ne^{-\Vert z\Vert^2}\overline{z}_j\alpha_j\right)=0\}
\end{equation*}
and the complex tangent space $T_p^{1,0}({\partial D_{n,m}})$ is defined by
\begin{equation*}
T_p^{1,0}({\partial D_{n,m}})=\{(\alpha,\beta)\in\mathbb{C}^{n+m},\sum_{k=1}^m\overline{w}_{k}\beta_k+
\sum\limits_{j=1}^ne^{-\Vert z\Vert^2}\overline{z}_j\alpha_j=0\}.
\end{equation*}
By the homogeneity of the boundary of $D_{n,m}$ under its automorphism, we fix $p=(0,\ldots,0,1,0,\ldots,0)\in \partial D_{n,m}$. Then we have
$$T_p(\partial D_{n,m})=\{(\alpha,\beta)\in\mathbb{C}^{n+m},\textrm{Re}\beta_1=0\}$$
and
$$T_p^{1,0}(\partial D_{n,m})=\{(\alpha,\beta)\in\mathbb{C}^{n+m},\beta_1=0\}.$$
Clearly we get
$$T_p^{1,0}(\partial D_{1,1})=\{(\alpha,\beta)\in\mathbb{C}^{2},\beta=0\}.$$

\begin{proof}[Proof of Theorem \ref{Th1.8}] Let
$$F=(f,h):=(f_1(z,w),\cdots, f_n(z,w); h_1(z,w),\cdots, h_m(z,w)):D_{1,1}\rightarrow D_{n,m}$$
 be a holomorphic mapping and holomorphic at $p$ with $F(p)=q$,
where $p=(0,1)\in \partial D_{1,1}$ and $q=(0,\ldots,0;1,0,\ldots,0)\in \partial D_{n,m}$.
In fact, since $F$ is holomorphic at $p$, there exists a neighborhood $V$ of $p$ such that $F$ is holomorphic in $V$.

For any $\alpha\in T_p(\partial D_{1,1})$, we can choose a smooth curve $\gamma:[-1,1]\rightarrow \partial D_{1,1}$ such that $\gamma(0)=p$, $\gamma'(0)=\alpha$, and $\gamma([-1,1])\subset V$. Hence,
$$\max\limits_{t\in[-1,1]}\rho_{D_{n,m}}(F(\gamma(t)))=0=\rho_{D_{n,m}}(F(\gamma(0))),$$
which implies
$$\frac{d}{d t}\rho_{D_{n,m}}(F(\gamma(t)))\vert_{t=0}=2\textrm{Re}\frac{\partial \rho_{D_{n,m}}}{\partial z}(F(p))J_F(p)\alpha^T=0.$$
This is,
\begin{equation*}
J_F(p)\alpha^T\subseteq (T_q(\partial D_{n,m}))^T.
\end{equation*}

Notice that for $\alpha\in T^{1,0}_p(\partial D_{1,1})$, we have $e^{i\theta}\alpha\in T^{1,0}_p(\partial D_{1,1})(\subset  T_p(\partial D_{1,1}))$ for any $\theta\in \mathbb{R}$. Then we obtain
$$2\textrm{Re}(e^{i\theta}\frac{\partial \rho_{D_{n,m}}}{\partial z}(F(p))J_F(p)\alpha^T)=0$$
for any $\theta\in \mathbb{R},$   which means
$$\frac{\partial \rho_{D_{n,m}}}{\partial z}(F(p))J_F(p)\alpha^T=0.$$
That is,
\begin{equation}\label{5500}
J_F(p)(T^{1,0}_p(\partial D_{1,1}))^T\subseteq (T^{1,0}_q(\partial D_{n,m}))^T.
\end{equation}

Thus, by (\ref{5500}), we have
\begin{equation}\label{551}
\frac{\partial h_1}{\partial z}(p)=0.
\end{equation}
On the other hand, define the holomorphic function $g(\zeta ):=h_1(\zeta p):E\rightarrow E$ (where E is the open unit
disk). Thus $g$ is holomorphic at $\zeta=1$ with $g(1)=1$. Hence, by Theorem \ref{Garnett},  we get
\begin{equation}\label{552}
\lambda:=\frac{\partial h_1}{\partial w}(p)=g'(1)\geq\frac{\vert1-\overline{h_1(0)}\vert^2}{1-\vert h_1(0)\vert^2}>0.
\end{equation}
Together with (\ref{551}) and (\ref{552}), we see that the real number $\lambda$ satisfies
\begin{equation*}
\overline{J_F(p)}^{T}q^{T}=\lambda p^T.
\end{equation*}

Now we consider the vector $\beta=(\alpha,0)\in T_p^{1,0}(\partial D_{1,1})$. It is well known that (see \cite{Royden})
\begin{equation}\label{eq25}
\mathcal{K}_{D_{1,1}}(bp,\beta) \geq  \mathcal{K}_{D_{n,m}}(F(bp),\beta (J_F(bp))^T)
\end{equation}
for any $0\leq b<1$.
Denote $\Omega :=\{(z,w)\in\mathbb{C}^{n}\times\mathbb{B}^{m}, \Vert z\Vert^2+\sum_{j=2}^{m}\vert w_j\vert^2+\ln(\vert w_1\vert^2)<0\}$. Thus we have
\begin{equation*}
D_{n,m}\subset\Omega\subset D_{n+m-1,1}.
\end{equation*}
This means
\begin{equation}\label{eq250}
\begin{aligned}
\mathcal{K}_{D_{n,m}}(F(bp),\beta (J_F(bp))^T)&\geq\mathcal{K}_{\Omega}(F(bp),\beta (J_F(bp))^T)\\
&\geq\mathcal{K}_{D_{n+m-1,1}}(F(bp),\beta (J_F(bp))^T).\\
\end{aligned}
\end{equation}
Now take the automorphism of $D_{n,m}$ as follows:
\begin{equation*}
\Phi:(z,w)\mapsto(z-f(bp),e^{\langle z,f(bp)\rangle-\frac{1}{2}{\Vert f(bp)\Vert}^{2}}w ).
\end{equation*}
Then we have
$$\Phi(F(bp))=(0,\ldots,0,e^{\frac{1}{2}{\Vert f(bp)\Vert}^{2}}h(bp)),$$
\begin{equation*}
J_{\Phi\circ{F}}(bp)=
\left(
  \begin{array}{cc}
    I_n & 0 \\
    e^{\frac{1}{2}{\Vert f(bp)\Vert}^{2}}{\overline{f(bp)}}^Th(bp) & e^{\frac{1}{2}{\Vert f(bp)\Vert}^{2}}I_m\\
  \end{array}
\right).
\end{equation*}
Therefore, we get
\begin{equation}\label{571}
\begin{aligned}
\mathcal{K}_{D_{n+m-1,1}}(F(bp),\beta(J_F(bp))^T)&=\mathcal{K}_{D_{n+m-1,1}}(\Phi\circ F(bp),\beta (J_{\Phi\circ F}(bp))^T)\\
&\geq\mathcal{K}_{D_{1,1}}((0,X(b)),(\vert {t}\vert,e^{i\theta} Y(b))),
\end{aligned}
\end{equation}
where
\begin{equation*}
\left\{
\begin{aligned}
X(b)&=e^{\frac{1}{2}({\Vert f(bp)\Vert}^{2}+\sum_{j=2}^{m}\vert h_j(bp)\vert^2)} \vert {h_1}(bp)\vert,\\
\vert {t}\vert^2&=\vert\alpha\vert^2(\Vert\frac{\partial F}{\partial z}(bp)\Vert^2-\vert\frac{\partial h_1}{\partial z}(bp)\vert^2),\\
Y(b)&=\alpha  e^{\frac{1}{2}({\Vert f(bp)\Vert}^{2}+\sum_{j=2}^{m}\vert h_j(bp)\vert^2)}\left( h_1(bp)(\sum_{k=1}^{n}\overline{f_k(bp)}\frac{\partial f_k}{\partial z}(bp)+\sum_{k=2}^{m}\overline{h_k(bp)}\frac{\partial h_k}{\partial z}(bp))+\frac{\partial h_1}{\partial z}(bp)\right).
\end{aligned}
\right.
\end{equation*}
Now together with \eqref{eq25}, \eqref{eq250} and \eqref{571}, we get
\begin{equation}\label{572}
\mathcal{K}_{D_{1,1}}((0,X(b)),(\vert {t}\vert,e^{i\theta} Y(b)))\leq\mathcal{K}_{D_{1,1}}((0,b),(\alpha,0)).
\end{equation}
Actually we can assume that $\vert {t}\vert^2>0$. Then (see \eqref{def000} for reference)
\begin{equation*}
v:=\frac{\vert Y(b)\vert^2}{\vert {t}\vert^2}\leq 2e^{({\Vert f(bp)\Vert}^{2}+\sum_{j=2}^{m}\vert h_j(bp)\vert^2)}\left(\vert h_1(bp)\vert^2(\Vert f(bp)\Vert^2+\sum_{j=2}^{m}\vert h_j(bp)\vert^2)+\frac{\vert\frac{\partial h_1}{\partial z}(bp)\vert^2}{\Vert\frac{\partial F}{\partial z}(bp)\Vert^2-\vert\frac{\partial h_1}{\partial z}(bp)\vert^2}\right).
\end{equation*}
Since $f(bp)\rightarrow 0$ and $\frac{\partial h_1}{\partial z}(bp)\rightarrow 0$ as $b\rightarrow 1^-$, we have $v\rightarrow 0$ as $b\rightarrow 1^-$. Thus there exist $\delta$ such that $v<4b^2$ for $b\in(\delta,1)$.
Thus,  by putting \eqref{eq22} into \eqref{572}, we have
\begin{equation}\label{5710}
\frac{1}{-2\ln X(b)}\bigg(\vert {t}\vert^2-\frac{\vert Y(b)\vert^2}{2X^{2}(b)\ln X(b)}\bigg)\leq\frac{\vert \alpha\vert^2}{-2\ln b}.
\end{equation}
Note $X(b)\rightarrow 1,\; Y(b)\rightarrow 0, \; \frac{\ln X^2(b)}{\ln b^2}\rightarrow\frac{g'(1)+\bar{g}'(1)}{2}=\lambda>0$ as $b\rightarrow 1^-$, and
$-\frac{\vert Y(b)\vert^2}{2X^{2}(b)\ln X(b)}\geq 0$ (note $0\leq X(b)<1$ for $0\leq b<1$ by $(0,X(b))\in D_{1,1}$).
Thus, from \eqref{5710},  we obtain
\begin{equation*}
\sum_{k=1}^{n}\vert\frac{\partial f_k}{\partial z}(p)\vert^2+\sum_{\ell=2}^{m}\vert\frac{\partial h_k}{\partial z}(p)\vert^2\leq \lambda.
\end{equation*}
That is,
\begin{equation*}
\Vert J_F(p)\Vert_{op}^2=\Vert\frac{\partial F}{\partial z}(p)\Vert^2\leq\lambda.
\end{equation*}
\end{proof}

\vskip 8pt

\noindent\textbf{Acknowledgments}\quad  We sincerely thank the referees for his useful comments. The project is supported
by the National Natural Science Foundation of China (No. 11671306) and the Natural Science Foundation of Shandong Province, China (No. ZR2018BA015).

\addcontentsline{toc}{section}{References}
\phantomsection
\renewcommand\refname{References}

 \end{document}